\documentclass[11pt,tbtags]{amsart}

\usepackage{latexsym,enumerate}
\usepackage{amsmath,amsthm,amsopn,amstext,amscd,amsfonts,amssymb}
\usepackage{color}

\makeatletter
\def\th@definition{%
  \normalfont 
}
\def\th@plain{%
  \slshape 
}
\def\th@remark{%
  \normalfont 
  \thm@preskip\topsep
  \divide\thm@preskip\tw@
  \thm@postskip\thm@preskip
}
\makeatother

\numberwithin{equation}{section}
\theoremstyle{plain}

\newtheorem{theorem}{Theorem}

\newtheorem{proposition}[theorem]{Proposition}
\theoremstyle{remark}
\newtheorem{remark}[theorem]{Remark}
\theoremstyle{definition}
\newtheorem{definition}[theorem]{Definition}

\numberwithin{theorem}{section}

\def\Om{\Omega}

\def\R{\mathbb R}

\def\D{\nabla }

\def\vp{\varepsilon}

\def\W{W^{1,p}(\Om)}

\newcommand{\feps}{f_\varepsilon}
\newcommand{\ueps}{u_\varepsilon}
\newcommand{\Phieps}{\Phi_\varepsilon}

\newcommand{\disp}{\displaystyle}
\newcommand{\diw}{{\operatorname{div}}}
\newcommand{\meas}{\operatorname{meas}}
\newcommand{\aop}{\operatorname{\mathbf{a}}}
\newcommand{\med}{{\operatorname{med}}}
\newcommand{\sign}{{\operatorname{sign}}}

\def\m2{|\Omega | /2}

\begin{document}
\date{}
\title[]{Neumann problems for nonlinear elliptic equations with $L^1$ data}

\begin{abstract} In the present paper we prove existence  results for solutions to nonlinear elliptic Neumann problems whose prototype is
\begin{equation*} 
    \begin{cases}
      -\Delta_{p} u -\diw (c(x)|u|^{p-2}u)) =f   & \text{in}\ \Omega, \\ 
      \left( |\nabla u|^{p-2}\nabla u+ c(x)|u|^{p-2}u \right)\cdot\underline n=0 & \text{on}\ \partial \Omega \,,%
    \end{cases}%
\end{equation*}%
when $f$ is just a summable function. Our approach allows also to deduce a stability result for renormalized solutions and an existence result for operator with a zero order term.



\end{abstract}

\author{M.F. Betta}
\address{Maria Francesca Betta \hfill \break\indent Dipartimento per le
Tecnologie, \hfill\break\indent Universit\`a degli Studi di Napoli
Parthenope,\hfill\break\indent Centro Direzionale, Isola C4 80143 Napoli,
Italy}
\email{francesca.betta@uniparthenope.it}
\author{O. Guib\'e}
\address{Olivier Guib\'e \hfill \break\indent 
  Laboratoire de Math\'ematiques Rapha\"el Salem, \hfill\break\indent
UMR 6085 CNRS-Universit\'e de Rouen\hfill\break\indent
Avenue de l'Universit\'e, BP.12\hfill\break\indent
76801 Saint-\'Etienne-du-Rouvray, France
}
\email{olivier.guibe@univ-rouen.fr}
\author{A. Mercaldo}
\address{Anna Mercaldo \hfill \break\indent Dipartimento di Matematica e
Applicazioni ``R. Caccioppoli", \hfill\break\indent Universit\`a di Napoli
``Federico II",\hfill\break\indent Complesso Monte S. Angelo, Via Cintia,
80126 Napoli, Italy}
\email{mercaldo@unina.it}
\maketitle

\section{Introduction}

In the present paper we prove existence results for solutions to nonlinear elliptic Neumann problems 
whose prototype is
\begin{equation} \label{pb0}
    \begin{cases}
        -\Delta_{p} u -\diw (c(x)|u|^{p-2}u) =f   & \text{in}\ \Omega, \\[.1cm] 
      \left( |\nabla u|^{p-2}\nabla u+ c(x)|u|^{p-2}u \right)\cdot\underline n=0 & \text{on}\ \partial \Omega ,%
  \end{cases}%
\end{equation}%
where $\Omega$ is a  bounded domain of $\R^{N}$,
$N\geq 2$, with Lipschitz boundary, $1< p\le N$ ,
$\underline n$ is the outer 
unit normal to $\partial \Omega$, the datum $f$ belongs to $L^{1}(\Omega)$ and satisfies
the compatibility condition $\int_\Omega f =0$. Finally  the coefficient $c(x)$ 
belongs to an appropriate Lebesgue space.

When 
$c(x)=0$ and $f$ is an element of the dual space of the Sobolev space
$W^{1,p}(\Omega)$, the existence and uniqueness (up to additive
constants) of weak solutions to problem \eqref{pb0} is consequence of
the classical theory of pseudo monotone operators (cfr. 
\cite{LL}, \cite{Lions}). But if $f$ is just an $L^1-$function, and not more an element of the dual space of $W^{1,p}(\Omega)$, one has to give a meaning to the notion of solution.

\noindent When Dirichlet boundary conditions are prescribed, various
definitions of solution to nonlinear elliptic equations with
right-hand side in $L^1$ or measure have been introduced. In
\cite{BBGGPV}, \cite{Aglio}, \cite{LM}, \cite{murat94}  different notions of 
solution are defined even if they turn out to be equivalent, at least when the datum is an 
$L^{1}-$ function. The
study of existence or uniqueness for Dirichlet boundary value problems
has been the object of several papers. We just recall that the linear
case has been studied in \cite{St}, while the nonlinear case began to be
faced  in \cite{BG1} and \cite{BG2} and was  continued in various
contributions, including \cite{AM1}, \cite{BeGu},  \cite{BBGGPV},  \cite{BMMP1},  \cite{BMMP},  \cite{DMOP},  \cite{Aglio},   \cite{GM2},  \cite{GM1};
mixed boundary value problems have been also studied  (see \cite{BeGu}). 
In the present paper we refer to the so-called renormalized solutions
(see \cite{DMOP},  \cite{LM},  \cite{murat94}) whose
precise definition is recalled in Section 2.

The existence for Neumann boundary value problems  with $L^1-$ data
when $c=0$ has been  treated in various contests. In
\cite{AMST97}, \cite{Ciabr}, 
\cite{Droniou00}, \cite{DV} and \cite{Prignet97} the
existence of a distributional  solution which belongs to a suitable
Sobolev space  and which has null mean value is proved. Nevertheless
when $p$ is close to 1, i.e. $p\le 2-1/N$, the distributional solution
to problem \eqref{pb0} does not belong to a Sobolev space and in
general is not a summable function; this implies that its mean value
has not meaning. This difficulty is overcome in \cite{Rako} by
considering solutions $u$ which are not in $L^1(\Omega)$, but for
which $\Phi(u)$ is  in $L^1(\Omega)$, where
$\Phi(t)=\int_{0}^{t}\frac{ds}{(1+|s|)^\alpha}$ with appropriate $\alpha>1$. In \cite{ACMM} the case where 
both the datum $f$ and the domain $\Omega$ are not regular is studied and  solutions whose  median is equal to zero are 
obtained with a natural process of approximations. We recall that the median of $u$ is defined by 
\begin{equation}\label{-520bis}
    {\rm med } (u) = \sup \{t\in \R : \meas\{u>t\} \geq \meas(\Omega)/2\}\, .
\end{equation}
Neumann problems have been studied by a different point of view in \cite{FM2, FM}.

In this paper we face two difficulties: one  due to the presence
of the lower order term $-\diw (c(x)|u|^{p-2}u))$ 
and the other due to the  low integrability properties of the datum $f$.

Our main result is Theorem \ref{exist_renorm} which asserts the existence of a renormalized solution to
\eqref{pb0} having $\med(u)=0$.  Its proof, contained in Section 4, is based on an usual procedure of
approximation which consists by considering problems of type
\eqref{pb0} having smooth data which strongly converge  to $f$ in
$L^1$. For such a sequence of problems we prove in Section 3 an existence results for weak solutions which is obtained by using a fixed point arguments.
A priori estimates allow to prove that these weak solutions converge in some
sense to a function $u$ and a delicate procedure of passage to the limit allows to prove that $u$ is a  renormalized solution to  \eqref{pb}.

 In Section 5 we give a stability
result and we prove that, under larger assumptions on the summability
of $f$, a renormalized solution to \eqref{pb} is in turn a weak solution to the same problem. At last Section 6
is concerned with Neumann problems with a zero order term; adapting
the proof of Theorem \ref{exist_renorm} allows to derive an existence result for this type of operators.


\section{Assumptions and definitions}

Let us consider the following nonlinear elliptic  Neumann problem
\begin{equation} \label{pb}
\left\{
\begin{array}{lll}
-\mbox{div}\left( \aop\left( x, u, \nabla u\right)+ \Phi (x,u) \right) =f &  &
\text{in}\ \Omega, \\
 \left( \aop\left( x, u, \nabla u\right)+ \Phi (x,u) \right)\cdot\underline n=0& & \text{on}\ \partial \Omega ,%
\end{array}%
\right. 
\end{equation}%
where $\Omega $ is a connected open subset of $\mathbb{R}^{N}$, $N\ge
2$, having finite Lebesgue measure and Lipschitz boundary, $\underline
n$ is the outer unit normal to $\partial \Omega$.  We assume that $p$
is a real number such that  $1<p \leq N$ and 
\[
\aop
:\Omega \times\mathbb{R}\times \mathbb{R}^{N}\rightarrow
\mathbb{R}^{N}\, ,
\]
\[
\Phi:\Omega \times \mathbb{R}\rightarrow \mathbb{R}^N
\] are Carath\'{e}odory functions. Moreover $\aop$ satisfies:
\begin{equation}
\aop\left( x, s, \xi \right) \cdot \xi \geq \alpha \left\vert \xi
\right\vert ^{p},\quad \forall s\in\R,\ \forall \xi\in\R^{N},\
\text{a.e. in $\Omega$}\label{ell}
\end{equation}
where $\alpha>0$ is a given real number;
\begin{equation}
\left( \aop\left( x, s, \xi \right) -\aop\left( x,s, \eta \right)
\right) \cdot \left( \xi -\eta \right) \geq 0
\label{mon}
\end{equation}
$\forall s\in\R$, $\forall \xi,\eta\in\R^{N}$ with $\xi\neq \eta$ and
a.e. in $\Omega$;
\par
\noindent for any $k>0$ there exist  $a_{k}>0$ and  $b_k $ belonging
to  $L^{p'}(\Omega)$ such that
\begin{equation}
  \label{growth}
\left\vert \aop\left( x,s, \xi \right) \right\vert \leq a_{k}\left\vert
\xi \right\vert ^{p-1}+b_k(x),\quad \forall |s|<k,\ \forall
\xi\in\R^{N},\ \text{a.e. in $\Omega$.}
\end{equation}
\par
We assume that $\Phi$ satisfies the following growth condition
\begin{equation}
| \Phi( x,s)|\le c(x) ( 1+|s|^{p-1})   \label{growthphi}
\end{equation} 
$\forall s \in \mathbb{R}$, a.e. in $\Omega$,
 with $c\in L^{\frac{N}{p-1}}(\Om)$ if $p<N$ and  $c\in L^{q}(\Omega)$ with
 $q>N/(N-1)$ if $p=N$.
\par
Finally we assume that the datum $f$ is a measurable function in a Lebesgue space $L^r(\Omega)$, $1\le r\le +\infty$, which belongs to the dual space of the classical Sobolev space $W^{1,p}(\Om)$ or is just an $L^1-$ function. Moreover it satisfies  the compatibility condition
\begin{equation}
\int_{\Om}f\,dx=0.\label{comp}
\end{equation}

\par
As explained in the Introduction we deal with solutions whose median
is equal to zero.
Let us recall that if  $u$ is a measurable function, we denote  the
median of $u$ by
\begin{equation}
    \med (u) = \sup \left\{ t\in \R:\meas \{x\in \Om : u(x)>t\}
        >\frac{\meas(\Om)}{2} \right\}.
\end{equation}
Let us explicitely observe that if $\med(u)=0$ then
\begin{gather*}
    \meas \{x\in \Om : u(x)>0\}   \le\frac{\meas(\Om)}{2},
    \\
    \meas \{x\in \Om : u(x)<0\}  \le\frac{\meas(\Om)}{2} \,.
\end{gather*}
In this case a Poincar\'e-Wirtinger inequality holds (see e.g. \cite{Z}):
\begin{proposition}
If $u\in W^{1,p}(\Om) $, then
\begin{equation}
\|u-\med (u)\|_{L^p(\Omega)}\le C\|\D u \|_{(L^p(\Omega))^N} \label{poincare}
\end{equation}
where $C$ is a constant depending on $p$, $N$, $\Om$.
\end{proposition}

As pointed out in the Introduction, when the datum $f$ is not an element of the dual space of the classical Sobolev
space $W^{1,p}(\Om)$ or is just an $L^1$-function, the classical notion of weak
solution does not fit. We will refer to the notion of  renormalized solution to \eqref{pb} (see
\cite{DMOP,murat94} for elliptic equations with Dirichlet boundary conditions) which we give below.

In the whole paper, $T_{k}$, $k\ge 0$, denotes the truncation at height $
k$ that is $T_{k}(s)=\min(k,\max(s,-k))$, $\forall s\in\R$. 
\begin{definition}\label{defrenorm}
  A real function $u$ defined in $\Omega$ is a renormalized solution
  to \eqref{pb} if 
  \begin{gather}
    \label{def1}
    \text{ $u$ is measurable and finite almost everywhere in
      $\Omega$,}
    \\
    \label{def2}
    T_{k}(u)\in W^{1,p}(\Omega), \text{ for any $k>0$,}
    \\
    \label{def3}
    \lim_{n\rightarrow +\infty }\frac{1}{n} \int_{\{ x\in\Omega;\, |u(x)|<n\}}
    \aop(x,u,\nabla u) \nabla u \,dx = 0
  \end{gather}
  and if for every function $h$ belonging to $W^{1,\infty}(\R)$ with
  compact support  and for every $\varphi\in L^{\infty}(\Omega)\cap
  W^{1,p}(\Omega)$
  we have
  \begin{multline}
      \int_{\Om} h(u) \aop (x,u, \D u) \D\varphi
  dx + \int_{\Om} h'(u) \aop (x,u, \D u) \D u   \varphi
  dx  \label{def4}\\ 
 +   \int_{\Om} h(u)\Phi (x, u)  \D\varphi dx + \int_{\Om} h'(u)\Phi (x, u)  \D u \varphi dx =
 \int_\Om f\varphi h(u) dx.
  \end{multline}
\end{definition}

\begin{remark} 
   A renormalized solution  is not  an
  $L^1_{loc}(\Omega)$-function and therefore it has not a
  distributional gradient. Condition \eqref{def2} allows to define a
  generalized gradient of $u$ according to Lemma 2.1 of \cite{BBGGPV},
  which asserts the existence of a unique measurable function $v$
  defined in $\Omega$ such that $\nabla T_k(u)=\chi_{\{|u|<k  \}}v$
  a.e. in $\Omega$, $\forall k>0$. This function $v$ is the
  generalized gradient of $u$ and it is denoted by $\nabla u$. 

Equality \eqref{def4} is formally obtained by using in \eqref{pb} the
test function $\varphi h(u)$ and  by taking into account Neumann boundary
conditions. Actually in a standard way one can  check that every term
in \eqref{def4} is well-defined under the structural assumptions on
the elliptic operator. 
\end{remark}

\begin{remark}
  It is worth noting that growth assumption \eqref{growthphi} on
  $\Phi$ together with \eqref{def1}--\eqref{def3} allow to prove that
  any renormalized solution $u$ verifies 
  \begin{equation}
    \label{ermk1}
    \lim_{n\rightarrow +\infty} \frac{1}{n} \int_{\Omega}
    |\Phi(x,u)|\times |\nabla T_{n}(u)| dx = 0.
  \end{equation}
  Without loss of generality we can assume that $\med(u)=0$. 
   Growth assumption \eqref{growthphi}
implies that
\begin{equation*}
  \int_{\Om} |\Phi (x, u)|  \times |\D T_{n}(u)|  dx 
  \leq
  \frac{1}{n} \int_{\Omega} c(x) (1+|T_{n}(u)|)^{p-1} |\nabla
  T_{n}(u)| dx.
\end{equation*}
In the case $N>p$, using H\"older inequality we obtain
\begin{equation}\label{star}
  \begin{split}
    \int_{\Omega} c(x) & (1+|T_{n}(u)|)^{p-1}  |\nabla
    T_{n}(u)| dx \\
    & \leq C \| c\|_{L^{N/(p-1)}(\Omega)} (1+ \|
    T_{n}(u)\|^{p-1}_{L^{p^{*}}(\Omega)}) \| \nabla
    T_{n}(u)\|_{(L^{p}(\Omega))^{N}}.
  \end{split}
\end{equation}
Since ${\med(T_{n}(u))}=0$, by Poincar\'e--Wirtinger inequality, i.e. Proposition 2.1,
and Sobolev embedding theorem it follows that
\begin{equation*}
  \begin{split}
    \int_{\Omega} c(x) & (1+|T_{n}(u)|)^{p-1}  |\nabla
    T_{n}(u)| dx \\
    & \leq C \| c\|_{L^{N/(p-1)}(\Omega)} (1+
    \| \nabla
    T_{n}(u)\|^{p-1}_{(L^{p}(\Omega))^{N}}) \| \nabla
    T_{n}(u)\|_{(L^{p}(\Omega))^{N}}
  \end{split}
\end{equation*}
where $C>0$ is a generic constant independent of $n$. Therefore  Young inequality
leads to 
\begin{equation}\label{ermk2}
  \begin{split}
     \frac{1}{n} \int_{\Omega} c(x) &  (1+|T_{n}(u)|)^{p-1}  |\nabla
    T_{n}(u)| dx \\
    & \leq \frac{C}{n} \| c\|_{L^{N/(p-1)}(\Omega)} (1+
    \| \nabla
    T_{n}(u)\|^{p}_{(L^{p}(\Omega))^{N}} )
    \end{split}
\end{equation}
In the case
$N=p$ a similar inequality involving $ \| c\|_{L^{q}(\Omega)}$ with
$q>N/(N-1)$ occurs. 

Due to the
coercivity of the operator $\aop$ and to \eqref{def3} we have
\begin{equation*}
  \lim_{n\rightarrow+\infty}\frac{1}{n} \int_{\Omega} |\nabla
  T_{n}(u)|^{p} dx =0.
\end{equation*}
By \eqref{star} and \eqref{ermk2} we  conclude that \eqref{ermk1} holds.
\end{remark}

\section{A basic existence result for weak solutions}

In this section we assume more restrictive conditions on the right-hand side $f$, on $\Phi$ and 
on the operator $\aop$ in order to prove the existence of a weak
solution $u$ to problem \eqref{pb}, that is
  \begin{gather*}
    u \in W^{1,p}(\Omega), \\
    \int_{\Omega} \aop(x,u,\nabla u)\nabla v dx +\int_{\Omega}
    \Phi(x,u) \nabla v dx =\int_{\Omega} fv  dx
  \end{gather*}
  for any $v\in W^{1,p}(\Omega)$.

We assume
\begin{gather}
  \label{3eq1}
  f\in L^{r}(\Omega)\cap (W^{1,p}(\Omega))'  \\
  \label{3eq2}
  |\Phi(x,s)| \leq c(x) \quad \text{$\forall s\in\R$, a.e. in
    $\Omega$} 
\end{gather}
with $c\in L^{\infty}(\Omega)$. Moreover the operator $\aop$ satisfies
\begin{equation}
	\label{3eq3a}
\left( \aop\left( x, s, \xi \right) -\aop\left( x,s, \eta \right)
\right) \cdot \left( \xi -\eta \right) >0
\end{equation}
$\forall s\in\R$, $\forall \xi,\eta\in\R^{N}$ with $\xi\neq \eta$ and
a.e. in $\Omega$;
\begin{equation}
  \label{3eq3}
  |\aop(x,s,\xi)|\leq a_{0}(|\xi|^{p-1}+|s|^{p-1}) +a_{1}(x) \quad
  \forall s\in\R,\, \forall \xi\in \R^{N},\ \text{a.e. in $\Omega$},
\end{equation}
with $a_{0}>0$, $a_{1}\in L^{p'}(\Omega)$.


\begin{theorem} \label{exist_weak} Assume that
  \eqref{ell}, 
 \eqref{3eq1}--\eqref{3eq3} and   \eqref{comp} hold.
  There exists at least one weak solution $u$ to
  problem \eqref{pb} having $\med(u)=0$.
\end{theorem}

\begin{proof} The proof relies on a fixed point argument.
  \par
  Let $v\in L^p(\Om)$. Due to \eqref{ell}, 
  \eqref{3eq3a} and
  \eqref{3eq3}, $(x,\xi)\in \Omega\times \R^{N}\mapsto 
  \aop(x,v(x),\xi)$ is a strictly   monotone operator and verifies 
  \[
	  |\aop(x,v(x),\xi)|\leq a_{0}(|\xi|^{p-1}+|v(x)|^{p-1}) +a_{1}(x) \quad
   \forall \xi\in \R^{N},\ \text{a.e. in $\Omega$}.
  \]
  Since $\Phi(x,v(x))\in (L^{\infty}(\Omega))^{N}$, classical
  arguments (see e.g. \cite{LL}, \cite{Lions}) allow to deduce that there
  exists a unique $ u$ such that 
\begin{equation} \label{3eq4}
   u \in \W, \qquad  \med(u)=0
  \end{equation}
  and 
 \begin{equation}
  \int_{\Om} \aop(x,v,\D u)\D\varphi \, dx =\int_{\Om}f\varphi \, dx -
 \int_{\Om} \Phi(x,v)\D\varphi \, dx\, , \quad \forall \varphi \in W^{1,p}(\Om)\,.
\label{weak_v}
 \end{equation}
It follows that we can consider the functional $\Gamma :
L^p(\Om)\longrightarrow L^p(\Om) 
$ defined by 
\[
\Gamma (v)=u\, , \qquad \forall v\in L^p(\Omega),
\]
where $u$ is the unique element of $W^{1,p}(\Omega)$ verifying
\eqref{3eq4} and \eqref{weak_v}.
We now prove that $\Gamma$ is a continuous and compact operator. 
\par
Let us  begin by proving that $\Gamma$ is continuous.  
Let $v_n\in L^p(\Om)$ such that $v_n\rightarrow v$ in $L^p(\Om)$. Up
to a subsequence (still denoted by $v_n$) $v_n \rightarrow v$ 
a.e. in $\Om$. Let $u_n=\Gamma (v_n) $ belonging to $W^{1,p}(\Omega)$
such  that $\med(u_n)=0$ and such that \eqref{weak_v} holds with $v_{n}$ in place of
$v$.

\noindent Choosing $\varphi =u_n$ as test function in \eqref{weak_v} and using
\eqref{ell} we obtain that
\[
\alpha \int_{\Omega} |\nabla u_{n}|^{p} dx \leq \int_{\Omega} |
fu_{n}| dx + \int_{\Omega} |\Phi(x,v_{n}) \nabla u_{n}| dx .
\]
Since $\med(u_n)=0$, from Poincar\'e-Wirtinger inequality \eqref{poincare}, \eqref{3eq1} and \eqref{3eq2}
Young inequality and Sobolev embedding theorem lead to
\begin{equation}
  \label{3eq5}
  \int_{\Omega} |\nabla u_{n}|^{p} dx \leq M
\end{equation}
where $M>0$ is a constant independent of $n$. Using again \eqref{poincare}, 
 it follows that $u_{n}$ is bounded in
$W^{1,p}(\Omega)$.

As a consequence and in view of \eqref{3eq3},  there exists a
subsequence (still denoted by $u_n$), a measurable function $u$ and
a field $\sigma$ belonging to  $(L^{p'}(\Omega))^{N}$ such that 
\begin{gather}
  \label{conv1}  u_n\rightharpoonup  u\quad\text{\rm weakly  in  } \W,\\
        \label{conv3} u_n\to u\quad\text{ strongly in
    }L^{p}(\Omega),
    \\
   \label{conv2} u_n\to u\quad\text{\rm a.e. in }\Omega,\\
    \label{conv3a} \aop(x,v_{n},\nabla u_{n})\rightharpoonup \sigma\quad\text{
      weakly in } (L^{p'}(\Omega))^{N}.
 \end{gather}
Since $\med(u_{n})=0$ for any $n$ and since $u\in
W^{1,p}(\Omega)$ the point-wise convergence of $u_{n}$ to $u$ implies
that $\text{med}(u)=0$. 

To get the continuity of $\Gamma$ it remains  to prove that $u=\Gamma
(v)$ that is $u$ satisfies \eqref{weak_v}. Using \eqref{weak_v} with
$v_{n}$ in place of $v$ and the test function $u_{n}-u$ we have
\begin{multline}\label{new}
\int_{\Omega} \aop(x,v_n,\D u_n)(\D u_n -\D u) dx =
\int_{\Omega} f (u_{n}-u) dx \\
- \int_{\Omega} \Phi(x,v_{n})(\nabla
u_{n}-\nabla u) dx.
\end{multline}
The point-wise convergence of $v_{n}$ and assumption \eqref{3eq2}
imply that $\Phi(x,v_{n})$ converges to $\Phi(x,v)$ almost everywhere
in $\Omega$ and in $L^{\infty}$ weak-* as $n$ goes to
infinity. Therefore from \eqref{conv1} and \eqref{conv3}, passing to the limit in the right-hand side of \eqref{new},
we obtain
\begin{equation}\label{minty0}
\lim_{n\rightarrow +\infty} \int_{\Omega}\aop(x,v_n,\D u_n)(\D u_n -\D u) dx = 0.
\end{equation}
Let us recall the classical arguments, so-called Minty arguments, (see \cite{LL}, \cite{Lions}) which
allow to identify $\sigma$ with $\aop(x,v,\nabla u)$. Let $\phi$
belonging to $(L^{\infty}(\Omega))^{N}$. 
Due to assumption \eqref{3eq3} and the convergence of $v_{n}$ the
Lebesgue theorem shows that for any $t\in\R$
\begin{equation*}
  \aop(x,v_{n},\nabla u+t\phi) \rightarrow \aop(x,v,\nabla u+t\phi) \text{
    strongly in $(L^{p'}(\Omega))^{N}$}.
\end{equation*}
By \eqref{conv3a} and \eqref{minty0}, it follows that for any $t\in\R$
\begin{align*}
\lim_{n\rightarrow +\infty} \int_{\Om} [\aop(x,v_n,\D u_n)- & \aop(x,v_{n},\D u + t\phi)](\D
u_n -\D u-t\phi ) dx 
\\
& = \int_{\Omega} [\sigma - \aop(x,v,\nabla u+t\phi)] t\phi dx.
\end{align*}
Using the monotone character \eqref{3eq3a} of $\aop$ we obtain that
for any $t\neq 0$
\begin{equation*}
  \sign(t)\int_{\Omega} [\sigma - \aop(x,v,\nabla u+t\phi)] \phi dx \geq 0.
\end{equation*}
Since $\aop(x,v,\nabla u +t \phi)$ converges strongly to
$\aop(x,v,\nabla u)$ in
$(L^{p'}(\Omega))^{N}$ as $t$ goes to zero, 
letting $t\rightarrow 0$ in the above inequality leads to 
\begin{equation*}
  \int_{\Omega} [\sigma-\aop(x,v,\nabla u)] \phi dx =0 
\end{equation*}
for any $\phi$ belonging to $(L^{\infty}(\Omega))^{N}$. We easily conclude that
\begin{equation}
  \label{conv4}
  \sigma= \aop(x,v,\nabla u).
\end{equation}

By using \eqref{conv3a} and \eqref{conv4} we can pass to the limit as $n\to +\infty$ in
\eqref{weak_v} with $v_{n}$ in place of $v$ and we get
 \begin{equation}
 \int_{\Om} \aop(x,v,\D u)\D\varphi dx =\int_{\Om}f\varphi dx - \int_{\Om} \Phi(x,v)\D\varphi dx\,, \quad \forall \varphi\in W^{1,p}(\Omega)\,.
\notag
\end{equation}

Since there exists a unique weak solution to \eqref{weak_v} with
median equal to zero we obtain that the whole sequence $u_{n}$
converges to $u$ in $L^{p}(\Omega)$ and $u=\Gamma (v)$. It follows
that $\Gamma$ is continuous.
\par
Compactness of $\Gamma $ immediatly follows.  Indeed, thank to the assumptions, for any $ v\in  L^p(\Omega)$,  we have  
\[
\int_\Omega |\D u |^p dx \le C\,,
\]
where $C$ is a constant depending on $\alpha$, $a_0$, $a_1$,
$\|c\|_{L^{\infty}(\Omega)}$, $\Om$, $N$, $p$ and $f$. Then, using
Poincar\'e-Wirtinger inequality and Rellich theorem, $u=\Gamma(v)$ belongs to a
compact set of $L^p(\Om)$. By choosing a ball of $L^p(\Om)$, $ B_{L^p}(0,
r)$ such that  
\[
\Gamma \left( B_{L^p}(0, r) \right ) \subset B_{L^p}(0, r)\,,
\]
Leray-Schauder fixed point theorem ensures the existence of at least one fixed point.
\end{proof}

\section{Existence result for renormalized solutions}

In this section we prove our main result which gives the existence of a renormalized solution to problem \eqref{pb}.

\begin{theorem} \label{exist_renorm}
 Assume \eqref{ell}--\eqref{comp}. If the datum $f$ belongs to $ L^{1}(\Om)$, then there exists at least one renormalized solution $u$ to problem \eqref{pb} having $\med(u)=0$. 
\end{theorem}

\begin{proof}  The proof is divided into 7 steps. In a standard way we begin by 
introducing a sequence of approximate problems whose data are smooth enough and
converge in some sense to the datum $f$. Then we prove that the weak solutions
$\ueps$ to the approximate problems and their gradients
$\nabla \ueps $ satisfy a priori estimates; such estimates allow to prove that
$\ueps$ and $\nabla \ueps$ converge to a function $u$ and its gradient $\nabla u$
respectively. The final step consists in proving that $u$ is a renormalized
solution to \eqref{pb} by showing that it is possible to pass to the limit in
the approximate problems.  
\par\smallskip
\noindent{\sl Step 1. Approximate problems.}

\noindent  For $\vp >0$, let us define 
\[\aop_{\vp}(x ,s,\xi)=\aop(x, T_{\frac1\vp}(s), \xi) + \vp|\xi|^{p-2}\xi,
\]
\[\Phi_{\vp}(x ,s)=T_{\frac1\vp}(\Phi(x, s) )
\]
and $f_{\vp}\in L^{p'}(\Om) $ such that 
\begin{gather*}
	\int_{\Omega} f_\varepsilon\,  dx = 0, \\
f_{\vp} \rightarrow f\quad \text{strongly in} \>L^1(\Om),
\\
\|f_{\varepsilon}\|_{L^{1}(\Omega)} \leq \|f\|_{L^{1}(\Omega)}\,,\quad
\forall \varepsilon>0.
\end{gather*}
Let us denote by $u_\vp$ one weak solution belonging to $W^{1,p}(\Om)$ such that 
\[
\text{med} (u_\vp)=0
\]
and
 \begin{equation}
 \int_{\Om} \aop_\vp(x,u_\vp,\D u_\vp)\D\varphi dx +\int_{\Om} \Phi_\vp(x,u_\vp)\D\varphi dx
=\int_{\Om}f_\vp\varphi dx \,,
\label{appr_eps}
 \end{equation}
 for every $\varphi \in W^{1,p}(\Om)$. The existence of such a function $u_\vp$ follows from Theorem \ref{exist_weak}.
\par
\smallskip
\noindent {\sl Step 2. A priori estimates}
  
  \noindent Using $\varphi =T_k(u_\vp)$ for $k>0$, as test function in \eqref{appr_eps}
  we have 
\begin{gather*}
 \int_{\Om} \aop_\vp (x,u_\vp,\D u_\vp)\D T_k(u_\vp) dx +\int_{\Om} \Phi_\vp(x,T_k (u_\vp))\D T_k(u_\vp) dx
\\
= \int_{\Om}f_\vp T_k(u_\vp) dx. 
 \notag
\end{gather*}
 which implies, by \eqref{ell} and \eqref{growthphi},
 \begin{align*}
 \alpha \int_\Om |\D T_k(u_\vp)|^p dx \le  \int_\Om c(x) (1+
 |T_k(u_\vp)|^{p-1})|\D T_k (u_\vp)| dx +k \| f\|_{L^1(\Omega)}. 
 \notag
  \end{align*}
By Young inequality we get 
\begin{equation}
  \label{4eq0}
  \int_\Om |\D T_k(u_\vp)|^p dx \le M(k+k^p) 
\end{equation}
for a suitable positive constant $M$ which depends on the data, but  does not depend on $k$ and
$\varepsilon$. 

\noindent We deduce that, for every $k>0$,
 \[T_k(u_\vp) \text{ is bounded in } W^{1,p}(\Om).
 \] 
Moreover taking into account  \eqref{growth} and \eqref{4eq0}, we obtain that
for any $k>0$
\begin{equation*}
  \aop(x,T_{k}(u_{\varepsilon}),\nabla
  T_{k}(u_{\varepsilon})) \text{ is bounded in $(L^{p'}(\Omega))^{N}$}
\end{equation*}
uniformly with respect to $\varepsilon$.
Therefore there exists  a measurable function $\,u
:\Om \rightarrow \overline\R$ and for any $k>0$ there exists a function
$\sigma_{k}$ belonging  to $(L^{p'}(\Omega))^{N}$ 
such
that, up to a  subsequence still indexed by $\varepsilon$, 
\begin{gather}
  \label{rin_conv1}  u_\vp \rightarrow u \text{ a.e. in } \Om,
  \\
  \label{rin_conv2} T_k(u_\vp) \rightharpoonup T_k(u) \text{ weakly   in }
  W^{1,p}(\Om),
  \\
  \label{rin_conv3}   \aop (x,T_{k}(u_\vp),\D
  T_K(u_\vp))\rightharpoonup \sigma_k \text{ weakly  in } (L^{p'}(\Om))^N
  \quad\forall k>0.
\end{gather}

\noindent {\sl Step 3. The function $u$ is finite a.e. in $\Om$ and
  $\med(u)=0$.} 
    
Since $\med(u_{\varepsilon})=0$,  Poincar\'e-Wirtinger inequality allows
us to use a log-type estimate (see \cite{BeGu,BOP,Droniou00,DV} for similar
 non coercive  problems). We consider the function  
\begin{equation}
  \Psi_p(r)=\int_0^r\frac{1}{(1+|s|)^p} ds\,, \quad  \forall  r\in \R.
  \notag
\end{equation}
We observe that 
$\med(\Psi_p(u_\vp))=\med(u_\vp)=0$.
Using $\Psi_p(u_\vp)$ as test function in \eqref{appr_eps}, we get
\begin{align*}
  \int_{\Om} \aop_\vp(x,u_\vp, \D u_\vp) \frac {\D u_\vp}{(1+|u_\vp|)^p}
  dx +\int_{\Om} \Phi_\vp (x, u_\vp) \frac {\D u_\vp}{(1+|u_\vp|)^p}
  dx = \int_\Om f_\vp\Psi_p(u_\vp) dx. 
\end{align*}   
By ellipticity condition \eqref{ell}, growth condition
\eqref{growthphi} and since $\disp \| \Psi_p(u_\vp)\|_{L^{\infty}(\Omega)} \le
\frac{1}{p-1}$, we get  
\begin{align*}
  \alpha\int_{\Om}  \frac {|\D u_\vp|^p}{(1+|u_\vp|)^p}dx &\le
  \int_{\Om} c(x) (1+|u_\vp|^{p-1} )\frac {|\D u_\vp|}{(1+|u_\vp|)^p} dx
  + \frac{1}{p-1}\|f\|_{L^{1}(\Omega)}\\ 
  & \le  C\int_{\Om} c(x) \frac {|\D u_\vp|}{(1+|u_\vp|)} dx +
  \frac{1}{p-1}\|f\|_{L^{1}(\Omega)}, 
\end{align*}   
where $C$ is a generic and positive constant independent of $\varepsilon$.
By Young inequality we deduce 
\begin{equation}
  \int_{\Om}  \frac {|\D u_\vp|^p}{(1+|u_\vp|)^p}dx \leq  Cp'
  \|c\|_{L^{p'}(\Omega)}^{p'}+ \frac{p'}{p-1}\|f\|_{L^1(\Omega)} \label{stima}. 
\end{equation}
Let us define
\[
\Psi_1(u_\vp)=\int_0^{u_\vp}\frac{1}{(1+|s|)}ds=\sign(u_{\varepsilon})
\ln(1+|u_{\varepsilon}|) .
\]
By \eqref{stima} we have
\[
	\|\D\Psi_1(u_\vp)\|_{(L^p(\Omega))^N}\le C
\]
and since $\med(\Psi_1(u_\vp))=0$,  Poincar\'e-Wirtinger inequality
leads to
\[
	\|\Psi_1(u_\vp)\|_{L^p(\Omega)}\le C.
\]
According to the definition of $\Psi_{1}$ we obtain that
\begin{equation}\label{eqog0}
\sup_{\varepsilon>0}
\meas(\{x\in\Omega\,;\,|u_{\varepsilon}(x)|>A\})\leq \frac{C}{\ln(1+A)}
\end{equation}
and this implies that $u$ is finite almost everywhere in $\Omega$.
\par
Since $\med(\ueps)=0$ for any $\varepsilon>0$ we also have, for any
$k>0$, $\med(T_{k}(\ueps))=0$, for any $\varepsilon>0$. Due to the
point-wise convergence of $\ueps$ and to the fact that $T_{k}(u)\in
W^{1,p}(\Omega)$ we obtain that $\med(T_{k}(u))=0$ for any $k>0$. It
follows that $\med(u)=0$.
\par\smallskip
\noindent{\sl Step 4. }{ \sl We  prove}
\begin{equation} \label{4eq00}
  \lim_{n\rightarrow +\infty}\limsup_{\varepsilon\rightarrow 0}
  \frac{1}{n} \int_{\Omega}
  \aop_{\varepsilon}(x,u_{\varepsilon},\nabla u_{\varepsilon}) \nabla
  T_{n}(u_{\varepsilon}) dx = 0.
\end{equation}
Using the test function $\frac{1}{n}T_{n}(u_{\varepsilon})$ in \eqref{appr_eps}
we have
\begin{gather*}
 \frac{1}{n}\int_{\Om} \aop_\vp (x,u_\vp,\D u_\vp)\D T_n(u_\vp) dx +\frac{1}{n}\int_{\Om} \Phi_\vp(x,T_n (u_\vp))\D T_n(u_\vp) dx
\\
=\frac{1}{n} \int_{\Om}f_\vp T_n(u_\vp) dx\,, 
 \notag
\end{gather*}
which yields that
\begin{equation}
  \label{eqog1}
  \begin{split}
  \frac{1}{n}\int_{\Om} \aop_\vp (x,u_\vp,\D u_\vp) & \D T_n(u_\vp) dx
  \leq \frac{1}{n} \int_{\Om}|f_\vp|\times  |T_n(u_\vp)| dx
  \\
  & {} + \frac{1}{n}\int_{\Om} c(x) (1+|T_n (u_\vp)|^{p-1}) |\D T_n(u_\vp)| dx.
\end{split}
\end{equation}
Due to \eqref{rin_conv1} the sequence $T_{n}(u_\varepsilon)$ converges to
$T_{n}(u)$ as $\varepsilon$ goes to zero in $L^{\infty}(\Omega)$ weak-*. Since
$f_{\varepsilon}$ strongly converges to $f$ in $L^{1}(\Omega)$ it follows that 
\begin{equation*}
  \lim_{\varepsilon\rightarrow 0}\frac{1}{n}
   \int_{\Om}|f_\vp|\times  |T_n(u_\vp)| dx=
  \frac{1}{n}
   \int_{\Om}|f|\times  |T_n(u)| dx.
\end{equation*}
Recalling that $u$ is finite almost everywhere in $\Omega$, the
sequence $T_{n}(u)/n$ converges to 0 as $n$ goes to infinity in
$L^{\infty}(\Omega)$ weak-*. Therefore we deduce that
\begin{equation}
  \label{eqog2}
  \lim_{n\rightarrow+\infty}\lim_{\varepsilon\rightarrow 0}\frac{1}{n}
  \int_{\Om}|f_\vp|\times  |T_n(u_\vp)| dx= 0.
\end{equation}
If $R$  is a positive real number which will be chosen later, let us define for
any $\varepsilon>0$ the
set $E_{\varepsilon,R}=\{ x\in \Omega\,:\, |u_\varepsilon(x)| > R\}$. 
We have for any $n>R$
\begin{equation}\label{eqog3}
  \begin{split}
    \frac{1}{n}\int_{\Om} c(x) & (1+|T_n (u_\vp)|^{p-1})  |\D T_n(u_\vp)| 
    dx \\
    \leq {}& \frac{1}{n}\int_{\Omega\setminus E_{\varepsilon,R}} c(x) (1+|
    T_{R}(u_\vp)|^{p-1}) |\D T_{R}(u_\vp)|    dx 
    \\
    & {} + \frac{1}{n}\int_{ E_{\varepsilon,R}} c(x) (1+|T_n
    (u_\vp)|^{p-1}) |\D T_n(u_\vp)|    dx. 
  \end{split}
\end{equation}
H\"older inequality yields that
\begin{equation*}
	\begin{split}
		\frac{1}{n}\int_{\Omega\setminus E_{\varepsilon,R}} & c(x) (1+|
    T_{R}(u_\vp)|^{p-1})  |\D T_{R}(u_\vp)|    dx 
    \leq \frac{1+R^{p-1}}{n} \int_{\Omega} c(x) |\nabla T_R(u_\varepsilon)| dx 
    \\
    & {} \leq \frac{1+R^{p-1}}{n} \| c\|_{L^{p'}{(\Omega)}} \|\nabla T_R(u_\varepsilon) \|_{(L^p(\Omega))^N}
    \end{split}
\end{equation*}
and since $T_R(u_\varepsilon)$ is bounded in $W^{1,p}(\Omega)$ uniformly with respect to $\varepsilon$
we obtain 
\begin{equation}\label{eqog4}
	\lim_{n\rightarrow+\infty} \limsup_{\varepsilon\rightarrow 0}
    \frac{1}{n}\int_{\{|u_{\varepsilon}|\leq R\}} c(x) (1+|
    T_{R}(u_\vp)|^{p-1})  |\D T_{R}(u_\vp)|    dx 
=0.
\end{equation}
To control the second term of the right-hand side of \eqref{eqog3} we distinguish
the case $p<N$ and $p=N$. If $p<N$ we have
\[
\frac{p-1}{N}+\frac{(N-p)(p-1)}{Np}+\frac{1}{p}=1
\]
so that  H\"older inequality gives
\begin{equation*}
	\begin{split}  
		\frac{1}{n}\int_{E_{\varepsilon,R}} & c(x) (1+|T_n
     (u_\vp)|^{p-1}) |\D T_n(u_\vp)|    dx \leq  
     \frac{1}{n} \| c\|_{L^{N/(p-1)}(E_{\varepsilon,R})}  
     \\
     & \times \Big( \meas(\Omega)^{Np/((N-p)(p-1))} + \| T_n(u_\varepsilon) \|_{L^{pN/(N-p)}(\Omega)} \Big) \| \nabla T_n(u_\varepsilon) \|_{(L^p(\Omega))^N}.
 \end{split}
\end{equation*} 
Recalling that $\med(T_n(u_\varepsilon))=0$  Poincar\'e-Wirtinger inequality and Sobolev embedding theorem lead to
\begin{equation}\label{eqog5}
\begin{split}
	\frac{1}{n}\int_{E_{\varepsilon,R}}  c(x) & (1+|T_n
     (u_\vp)|^{p-1})  |\D T_n(u_\vp)|    dx\\ & \leq \frac{C}{n} 
     \| c\|_{L^{N/(p-1)}(E_{\varepsilon,R})}   
     \Big(  1  
     + \| \nabla T_n(u_\varepsilon) \|_{(L^p(\Omega))^N}^p
     \Big)
\end{split}
\end{equation}
where $C>0$ is a constant independent of $n$ and $\varepsilon$.
If $p=N$, since $c$ belongs to $L^q(\Omega)$ with $q>\frac{N}{N-1}$ similar
arguments lead to
\begin{equation}\label{eqog6}
	\begin{split}
		\frac{1}{n}\int_{E_{\varepsilon,R}}  c(x) &(1+|T_n
     (u_\vp)|^{p-1})  |\D T_n(u_\vp)|    dx \\
     & \le \frac{C}{n} 
     \| c\|_{L^{q}(E_{\varepsilon,R})}  \Big(  1  
     + \| \nabla T_n(u_\varepsilon) \|_{(L^p(\Omega))^N}^p
     \Big)
\end{split}
\end{equation}
where $C>0$ is a constant independent of $n$ and $\varepsilon$.
\par
In view of \eqref{eqog0} and the equi-integrability of $c$ in $L^{q}
(\Omega)$ (with  $q=N/(p-1)$ if $p<N$  and $q>N/(N-1)$ if $p=N$) 
let $R>0$ such that for any $\varepsilon>0$
\begin{equation}\label{eqog6b}
	C \| c\|_{L^{q}(E_{\varepsilon,R})}  < \frac{\alpha}{2}\,,
\end{equation}
where $\alpha$ denotes the ellipticity constant in \eqref{ell}. 
Using the ellipticity condition \eqref{ell} together with
\eqref{eqog1}--\eqref{eqog6b} leads to
\[
  \frac{1}{n}\int_{\Om} \aop_\vp (x,u_\vp,\D u_\vp)  \D T_n(u_\vp) dx
  \leq \frac{C}{n} \|c\|_{L^q(\Omega)}+\omega(\varepsilon,n) 	
\]
with $q=N/(p-1)$ if $p<N$ and $q>N/(N-1)$ if $p=N$ and where
$\omega(\varepsilon,n)$ is such that
$\lim_{n\rightarrow\infty}\limsup_{\varepsilon\rightarrow0}
\omega(\varepsilon,n)=0$. 
\par
It follows that \eqref{4eq00} holds.

\par\smallskip
\noindent{\sl Step 5. We  prove that for any $k>0$}
\begin{equation}
  \label{4eq1}
  \begin{split}
  \lim_{\varepsilon \rightarrow 0} \int_{\Omega}
   (\aop(x,T_{k}(u_\varepsilon)&,\nabla
  T_{k}(u_\varepsilon))-\aop(x,T_{k}(\ueps),\nabla T_{k}(u))) \\
  & \cdot (\nabla
  T_{k}(u_{\varepsilon})-\nabla T_{k}(u)) dx =0.
  \end{split}
\end{equation} 
Let $h_{n}$ defined by 
\begin{equation}  \label{h_n}
h_n(s)=
   \begin{cases}
       0    &\quad \text{if } |s|>2n, \\
\displaystyle\frac{2n-|s|}{n}     &\quad  \text{if } n<|s|\le 2n,\\
1     &\quad \text{ if } |s|\le n\,.
 \end{cases}
\end{equation}
Using the admissible test function $h_{n}(u_{\varepsilon})
(T_{k}(u_{\varepsilon})-T_{k}(u))$ to \eqref{appr_eps} we have
\begin{equation}
  \label{4eq2}
  \begin{split}
  \int_{\Omega} h_{n}(u_{\varepsilon}) & \aop(x,u_{\varepsilon},\nabla
  u_{\varepsilon})  (\nabla T_{k}(u_{\varepsilon})- \nabla T_{k}(u)) dx \\
  & {} = A_{k,n,\varepsilon} +  B_{k,n,\varepsilon} +
  C_{k,n,\varepsilon} +  D_{k,n,\varepsilon} +   E_{k,n,\varepsilon}
  \end{split}
  \end{equation}
  with
  \begin{gather*}
  A_{k,n,\varepsilon} =\int_{\Omega} h_{n}(\ueps) f_{\varepsilon}
  (T_{k}(\ueps)-T_{k}(u))dx,
 \\
 B_{k,n,\varepsilon}= - \int_{\Omega} h_{n}(\ueps) \Phi_{\varepsilon}(x,\ueps) (\nabla
 T_{k}(\ueps)-\nabla T_{k}(u)) dx,
 \\
 C_{k,n,\varepsilon}= - \int_{\Omega} h'_{n}(\ueps) \Phi_{\varepsilon}(x,\ueps) \nabla \ueps
 (T_{k}(\ueps)-T_{k}(u)) dx,
  \\
 D_{k,n,\varepsilon} = - \int_{\Omega} h'_{n}(\ueps)
 \aop_{\varepsilon}(x,u_\varepsilon,\nabla
 u_\varepsilon) \nabla \ueps  (T_{k}(u_{\varepsilon}) - T_{k}(u))dx,
 \\
 E_{k,n,\varepsilon}=  -\varepsilon \int_{\Omega} h_{n}(\ueps) |\nabla \ueps |^{p-2} \nabla
 \ueps ( \nabla   T_{k}(u_{\varepsilon}) -\nabla
 T_{k}(u))dx.
\end{gather*}
We now pass to the limit in \eqref{4eq2} first as $\varepsilon$ goes
to zero and then as $n$ goes to infinity.
\par
Due to the point-wise convergence of $\ueps$ the sequence
$T_{k}(\ueps)-T_{k}(u)$ converges to zero almost everywhere in
$\Omega$ and in
$L^{\infty}(\Omega)$ weak* as $\varepsilon$ goes to zero. Since
$f_{\varepsilon}$ converges to $f$ 
strongly in $L^{1}(\Omega)$ we obtain that
\begin{equation*}
  \lim_{\varepsilon \rightarrow 0}  A_{k,n,\varepsilon} = \lim_{\varepsilon\rightarrow 0} \int_{\Omega} h_{n}(u_{\varepsilon}) f_{\varepsilon}
  (T_{k}(\ueps)-T_{k}(u))dx =0.
\end{equation*}
For $\varepsilon<1/n$  we have
$h_{n}(s)\Phi_{\varepsilon}(x,s)=h_{n}(s)\Phi(x,s)$ for any $s\in\R$
and a.e. in $\Omega$. Using the point-wise convergence of $\ueps$ 
  $h_{n}(\ueps) \Phi_{\varepsilon}(x,\ueps)$ converges to
$h_{n}(u)\Phi(x,u)$ a.e. in $\Omega$ as $\varepsilon$ goes to zero while by
\eqref{growthphi} we have
$h_{n}(u_{\varepsilon})|\Phi_{\varepsilon}(x,\ueps)|\leq (1+(2n)^{p-1}) c(x)$. It
follows that $h_{n}(u_{\varepsilon})\Phi_{\varepsilon}(x,\ueps)$
converges to $h_{n}(u)\Phi(x,u)$ strongly in $(L^{q}(\Omega))^{N}$
with $q=N/(p-1)$ if $N>p$ and $q>N/(N-1)$ if $N=p$. Due to
\eqref{rin_conv2} we deduce that
\begin{equation*}
  \lim_{\varepsilon\rightarrow 0} B_{k,n,\varepsilon}= - \lim_{\varepsilon\rightarrow 0} \int_{\Omega} h_{n}(\ueps)
  \Phi_{\varepsilon}(x,\ueps) (\nabla 
 T_{k}(\ueps)-\nabla T_{k}(u)) dx =0.
\end{equation*}
With  arguments already used we also have for any $n\geq 1/\varepsilon$
\begin{equation*}
 \lim_{\varepsilon\rightarrow 0} C_{k,n,\varepsilon}= - 
  \lim_{\varepsilon\rightarrow 0}   \int_{\Omega} h'_{n}(\ueps)
\Phi_{\varepsilon}(x,\ueps) \nabla \ueps 
 (T_{k}(\ueps)-T_{k}(u)) dx =0.
\end{equation*}
Since 
\begin{equation*}
|D_{k,n,\varepsilon}|
 \leq
 \frac{2k}{n} \int_{\{|\ueps|\leq 2n\}} \aop_{\varepsilon}(x,u_\varepsilon,\nabla
 u_\varepsilon) \nabla \ueps dx
\end{equation*}
and due to \eqref{4eq00} we obtain that
\begin{equation*}
  \lim_{n\rightarrow 0}\limsup_{\varepsilon\rightarrow 0 } D_{k,n,\varepsilon}=0.
\end{equation*}
The identification $h_{n}(\ueps) |\nabla \ueps |^{p-2} \nabla
\ueps = h_{n}(\ueps) |\nabla T_{2n}(\ueps) |^{p-2} \nabla
T_{2n}(\ueps)$ a.e.  in $\Omega$ and  estimate \eqref{4eq0} imply that
\[
h_{n}(\ueps) |\nabla \ueps |^{p-2} \nabla
\ueps
( \nabla   T_{k}(u_{\varepsilon}) -\nabla
T_{k}(u))
\]
is bounded in $L^{1}(\Omega)$ uniformly with respect to $\varepsilon$.
It follows that 
\[
	\varepsilon h_{n}(\ueps) |\nabla \ueps |^{p-2} \nabla
\ueps
( \nabla   T_{k}(u_{\varepsilon}) -\nabla
T_{k}(u))
\]
converges to 0 strongly in $L^1(\Omega)$ so that
\begin{equation*}
  \lim_{\varepsilon\rightarrow 0 } E_{k,n,\varepsilon}=0.
\end{equation*}

As a consequence we obtain that for any $k>0$
\begin{equation*}
 \lim_{n\rightarrow \infty}\limsup_{\varepsilon\rightarrow 0 }  \int_{\Omega} h_{n}(u_{\varepsilon})  \aop(x,u_{\varepsilon},\nabla
  u_{\varepsilon})  (\nabla T_{k}(u_{\varepsilon})- \nabla T_{k}(u))
  dx= 0.
\end{equation*}
Recalling that for any $n>k$, we have 
\[
	h_{n}(u_{\varepsilon})
\aop(x,u_{\varepsilon},\nabla   u_{\varepsilon}) \nabla
T_{k}(u_{\varepsilon})=
\aop(x,u_{\varepsilon},\nabla   u_{\varepsilon}) \nabla
T_{k}(u_{\varepsilon}) \quad\text{a.e. in $\Omega$.}
\]
It follows that 
\begin{multline}
  \label{4eq3}
  \limsup_{\varepsilon\rightarrow 0 }\int_{\Omega}  \aop(x,u_{\varepsilon},\nabla
  u_{\varepsilon})  \nabla T_{k}(u_{\varepsilon}) 
  dx 
  \\
  \leq  \lim_{n\rightarrow \infty}\limsup_{\varepsilon\rightarrow 0 }  \int_{\Omega} h_{n}(u_{\varepsilon})  \aop(x,u_{\varepsilon},\nabla
  u_{\varepsilon})   \nabla T_{k}(u)   dx.
\end{multline}
According to the definition of $h_{n}$ we have
\[
h_{n}(\ueps)  \aop(x,u_{\varepsilon},\nabla
  u_{\varepsilon}) = h_{n}(\ueps)  \aop(x,T_{2n}(u_{\varepsilon}),\nabla
  T_{2n}(u_{\varepsilon})) \text{ a.e. in $\Omega$}
\]
so that \eqref{rin_conv1} and \eqref{rin_conv3} give
\begin{equation}
  \label{4eq4}
\lim_{\varepsilon\rightarrow 0 }  \int_{\Omega} h_{n}(u_{\varepsilon})  \aop(x,u_{\varepsilon},\nabla
  u_{\varepsilon})   \nabla T_{k}(u)   dx = \int_{\Omega} h_{n}(u)
  \sigma_{2n} \nabla T_{k}(u) dx.
\end{equation}
\par
If $n>k$ we have
\[
  \aop(x,T_{n}(u_{\varepsilon}),\nabla 
  T_{n}(u_{\varepsilon})) \chi_{\{|\ueps|<k\}} =  \aop(x,T_{k}(u_{\varepsilon}),\nabla 
  T_{k}(u_{\varepsilon})) \chi_{\{|\ueps|<k\}} 
\]
almost everywhere in $\Omega$. From \eqref{rin_conv1} and
\eqref{rin_conv3} it follows that
\[
\sigma_{n} \chi_{\{|u|<k\}} = \sigma_{k} \chi_{\{|u|<k\}} 
\text{  a.e. in $\Omega\setminus\{|u|=k\}$} 
\]
and then we obtain for any $n>k$
\[
\sigma_{n} \nabla T_{k}(u)= \sigma_{k} \nabla T_{k}(u) \text{ a.e. in
  $\Omega$.}
\]
Therefore \eqref{4eq3} and \eqref{4eq4} allow to conclude that
\begin{equation}
  \label{4eq5}
  \limsup_{\varepsilon \rightarrow 0} \int_{\Omega}
  \aop(x,T_{k}(u_\varepsilon),\nabla T_{k}(u_\varepsilon)) \nabla
  T_{k}(u_{\varepsilon}) dx \leq \int_{\Omega} \sigma_{k} \nabla
  T_{k}(u) dx.
\end{equation}
We are now in a position to prove \eqref{4eq1}. Indeed 
the monotone character of $\aop$ implies that for any $\varepsilon>0$
\begin{multline}
  \label{4eq6}
  0 \leq 
   \int_{\Omega}
  (\aop(x,T_{k}(u_\varepsilon),\nabla
  T_{k}(u_\varepsilon))  -\aop(x,T_{k}(\ueps),\nabla T_{k}(u))) 
  \\
   \cdot (\nabla
  T_{k}(u_{\varepsilon})-\nabla T_{k}(u)) dx.\qquad\quad
\end{multline}

Moreover,  using the point-wise convergence of $T_{k}(\ueps)$
and assumption \eqref{growth}, the function
$\aop(x,T_{k}(\ueps),\nabla T_{k}(u))$ converges to
$\aop(x,T_{k}(u),\nabla T_{k}(u))$ strongly in $(L^{p'}(\Omega))^{N}$. Writing
\begin{equation*}
  \begin{split}
      \int_{\Omega}
  (\aop(x,T_{k}(u_\varepsilon), & \nabla
  T_{k}(u_\varepsilon))-\aop(x,T_{k}(\ueps),\nabla T_{k}(u))) (\nabla
  T_{k}(u_{\varepsilon})-\nabla T_{k}(u)) dx
    \\
{} = {}&   \int_{\Omega}
  \aop(x,T_{k}(u_\varepsilon),\nabla
  T_{k}(u_\varepsilon)) (\nabla
  T_{k}(u_{\varepsilon})-\nabla T_{k}(u)) dx
  \\
  & {} -  \int_{\Omega}
  \aop(x,T_{k}(\ueps),\nabla T_{k}(u)) (\nabla
  T_{k}(u_{\varepsilon})-\nabla T_{k}(u)) dx,
\end{split}
\end{equation*}
using \eqref{4eq5} and  \eqref{4eq6} allow to conclude that
\eqref{4eq1} holds for any $k>0$.

\par\smallskip

\noindent{\sl Step 6. We prove in this step that for any $k>0$}
\begin{gather}
  \aop(x,T_{k}(u),\nabla T_{k}(u)) = \sigma_{k} \label{4eq7}
\\ 
\aop(x,T_{k}(u_\varepsilon),  \nabla
  T_{k}(u_\varepsilon)) \nabla T_{k}(\ueps) \rightharpoonup
  \aop(x,T_{k}(u),\nabla T_{k}(u))\nabla T_{k}(u) \label{4eq8}
\end{gather}
weakly in $L^{1}(\Omega)$ as $\varepsilon$ goes to zero.
\par
From \eqref{4eq6} we have for any $k>0$
\begin{equation*}
   \lim_{\varepsilon \rightarrow 0} \int_{\Omega}
  \aop(x,T_{k}(u_\varepsilon),\nabla T_{k}(u_\varepsilon)) \nabla
  T_{k}(u_{\varepsilon}) dx = \int_{\Omega} \sigma_{k} \nabla
  T_{k}(u) dx.
\end{equation*}
The monotone character of $\aop$ and the usual Minty argument imply
\eqref{4eq7}.
\par
From \eqref{4eq1} we get  
\[
(\aop(x,T_{k}(u_\varepsilon),  \nabla
  T_{k}(u_\varepsilon))-\aop(x,T_{k}(\ueps),\nabla T_{k}(u))) (\nabla
  T_{k}(u_{\varepsilon})-\nabla T_{k}(u)) \rightarrow 0
\]
strongly in $L^{1}(\Omega)$ as $\varepsilon$ goes to zero.
Using \eqref{rin_conv2} and recalling that the sequence
$\aop(x,T_{k}(\ueps),\nabla 
T_{k}(u)))$ converges to $\aop(x,T_{k}(u),\nabla T_{k}(u)))$ strongly
in $(L^{p'}(\Omega))^{N}$ the monotone character of $\aop$ leads to \eqref{4eq8}.

\par\smallskip
\noindent {\sl  Step 7. We are now in a position to pass to the limit in the
	approximated problem.
}
\par
Let $h$ be a function in $W^{1,\infty}(\R)$ with compact support, contained in the interval $[-k,k]$, $k>0$ and let $\varphi\in W^{1,p}(\Omega)\cap L^{\infty}(\Om)$.
   Using $\varphi h(u_\vp)$ as a test function in the approximated problem we have 
   \begin{multline}
  \int_{\Om} h(u_\vp) \aop_{\varepsilon} (x,u_\vp, \D u_\vp) \D\varphi
  dx + \int_{\Om} h'(u_\vp) \aop_{\varepsilon} (x,u_\vp, \D u_\vp) \D u_\vp   \varphi
  dx  \label{rin_eps}\\ 
 +   \int_{\Om} h(u_\vp)\Phi_\vp (x, u_\vp)  \D\varphi dx + \int_{\Om} h'(u_\vp)\Phi_\vp (x, u_\vp)  \D u_\vp \varphi dx \\
 = {} \int_\Om f_\vp\varphi h(u_\vp) dx.
    \end{multline}
We want to pass to the limit in this equality. Since $\text{supp}\, h$
is contained in the interval  
$[-k,k]$, by the strong converge of $f_\vp$ to $f$ and \eqref{rin_conv1}
we immediatly obtain 
\[
\lim_{\vp\to0} \int_\Om f_\vp\varphi h(u_\vp) dx= \int_\Om f\varphi h(u) dx.
\]
Moreover by growth condition \eqref{growthphi} and  \eqref{rin_conv1},
using Lebesgue convergence theorem we deduce that 
\begin{equation}
\lim_{\vp\to0}\int_{\Om} h(u_\vp)\Phi_\vp (x, u_\vp)  \D\varphi
dx=\int_{\Om} h(u)\Phi (x, u)  \D\varphi dx.\notag 
\end{equation}
Analogously from \eqref{rin_conv2} we obtain
\begin{align*}
\lim_{\vp\to0}\int_{\Om} h'(u_\vp)\Phi_\vp (x, u_\vp)  \D u_\vp
\varphi dx & =\lim_{\vp\to0}\int_{\Om} h'(u_\vp)\Phi (x, T_{k}(u_\vp))  \D T_k(u_\vp ) \varphi dx\\
& = \int_{\Om} h'(u)\Phi (x, u)  \D T_k(u) \varphi dx.
\end{align*}
In view of the definition of $\aop_{\varepsilon}$ and since
$\varepsilon |\nabla T_{k}(\ueps)|^{p-2}\nabla T_{k}(\ueps)$ converges to
zero strongly in $(L^{p'}(\Omega))^{N}$ as $\varepsilon$ goes to zero,
 {\eqref{rin_conv3}} and \eqref{4eq7} imply that
\begin{align*}
\lim_{\vp\to0} \int_{\Om} h(u_\vp) \aop_{\varepsilon}(x,u_\vp, \nabla T_k(u_\vp)
) & \D\varphi dx \\
 & =\lim_{\vp\to0} \int_{\Om} h(u_\vp) \aop (x,T_{k}(u_\vp), \D T_k(u_\vp)
)\D\varphi dx 
\\
& = \int_{\Om} h(u) \aop (x,u, \D T_k(u) )\D\varphi dx.
\end{align*}
From \eqref{4eq8} we get 
\begin{align*}
\lim_{\vp\to0} \int_{\Om} h'(u_\vp) \aop_{\varepsilon} & (x,u_\vp, \D
u_\vp) \D u_\vp   \varphi dx  \\
& {} = \lim_{\vp\to0} \int_{\Om} h'(u_\vp) \aop (x,T_{k}(u_\vp) , \D T_k(u_\vp))
\D T_k(u_\vp))   \varphi dx \\ 
&{} = \int_{\Om} h'(u) \aop (x,u, \D T_k(u)) \D T_k(u)   \varphi dx. \notag
\end{align*}
Therefore by passing to the limit in \eqref{rin_eps} we obtain condition
\eqref{def4}  in the definition of renormalized solution. 
The decay of the truncated energy \eqref{def3} is a consequence of
\eqref{4eq00} and \eqref{4eq8}. Since $u$ is finite almost everywhere
in $\Omega$ and since $T_{k}(u)\in W^{1,p}(\Omega)$ for any $k>0$ we
can conclude that $u$ is a renormalized solution to \eqref{pb} and
that $\med(u)=0$.
\end{proof}

\section{Stability result and further remarks}

This section is devoted to  state a stability result and to prove
that if the right-hand side $f$ is  regular enough, under additional
assumptions on $\aop$,  then any renormalized  solution is also a weak solution.

For $\varepsilon>0$ let $\feps$ belonging to $L^{1}(\Omega)$ and
$\Phieps : \Omega\times \R \mapsto \R^{N}$ a Carath\'eodory function.
Assume that there exists $c\in L^{q}(\Omega)$ with $q=N/(p-1)$ if
$p<N$ and  $q>N/(N-1)$ if $p=N$ such that for any $\varepsilon>0$ 
\begin{equation}\label{stab00}
  \left\vert \Phieps(x,s)\right\vert \leq c(x)(\left\vert
    s\right\vert ^{p-1}+1)
\end{equation}%
for almost everywhere in $\Omega$ and every $s\in\R$.
For any $\varepsilon>0$ let $\ueps$ be a renormalized solution (having null median) to the problem 
\begin{equation} \label{stab1}
    \begin{cases}
      -\diw\left( \aop\left( x, \ueps, \nabla \ueps\right)+ \Phieps (x,\ueps)
      \right) =\feps & \text{in}\ \Omega,\\
      \left( \aop\left( x, \ueps, \nabla \ueps\right)+ \Phieps (x,\ueps) \right)\cdot\underline n=0 & \text{on}\ \partial \Omega ,%
    \end{cases}%
\end{equation}%
where $\aop$ verifies \eqref{ell}--\eqref{growth}.
\par
Moreover assume that 
\begin{equation}
    \int_\Omega f_\varepsilon\, dx=0,\qquad   f_{\varepsilon}\rightarrow f\text{\ strongly in }L^{1}(\Omega)  
  \label{stab2}
\end{equation}%
and for almost every $x$ in $\Omega$
\begin{equation}\label{stab3}
    \begin{cases}
      \Phieps(x,s_{\varepsilon})\rightarrow \Phi(x,s) \\
      \text{for every sequence }s_{\varepsilon}\in\mathbb{R}\text{ such that } 
      s_{\varepsilon}\rightarrow s
    \end{cases}
\end{equation}%
where $\Phi$ is a Carath\'eodory function verifying (as a consequence
of \eqref{stab00}) the growth condition \eqref{growthphi}).
\par

\begin{theorem}\label{thstab}
Under the assumptions \eqref{stab00}, \eqref{stab1}, \eqref{stab2}, \eqref{stab3},
 up to a subsequence (still indexed by $\varepsilon$)
$\ueps$ converges to $u$ as $\varepsilon$ goes to zero where $u$  is a
renormalized solution to \eqref{pb} with null median. More precisely
we have 
\begin{gather}
  u_{\varepsilon}\rightarrow u\text{ a.e in }\Omega,
  \label{stab4}
  \\
  \aop(x,T_{k}(\ueps),\nabla T_{k}(\ueps)) \nabla T_{k}(\ueps) 
  \rightharpoonup \aop(x,T_{k}(u),\nabla T_{k}(u)) \nabla T_{k}(u) 
  \label{stab5}
\end{gather}
weakly $L^{1}(\Omega)$.
\end{theorem}

\begin{proof}[Sketch of proof]
  We mainly follow the arguments developed in the proof of Theorem
  \ref{exist_renorm}. 
  As usual, the crucial point is to obtain 
  {\em a priori}
  estimates, i.e.
  \begin{gather}\label{stab6a}
    T_{k}(\ueps) \text{ bounded in } W^{1,p}(\Omega), \\
    \aop(x,T_{k}(\ueps),\nabla T_{k}(\ueps)) \text{ bounded in }
    (L^{p'}(\Omega))^{N}\text{ for any $k>0$} \label{stab6b}
  \end{gather}
  and
  \begin{gather}
\label{stab6d}
\lim_{n\rightarrow +\infty}\limsup_{\varepsilon\rightarrow 0}
  \frac{1}{n} \int_{\Omega}
  \aop(x,u_{\varepsilon},\nabla u_{\varepsilon}) \nabla
  T_{n}(u_{\varepsilon}) dx = 0.
  \end{gather}
   
  Even if $T_{k}(u_{\varepsilon })$ is not an admissible test function in the
renormalized formulation (see Definition \ref{defrenorm}) it is well known that it can be achieved through
the following process.  Using $h=h_{n}$, where $h_n$ is defined in \eqref{h_n},  and $\varphi= T_{k}(\ueps)$ in the
renormalized formulation \eqref{def4} we have, for any $n>0$ and any $k>0$
\begin{multline}
  \label{5eq1}
  \int_{\Om} h_{n}(\ueps) \aop (x,\ueps, \D \ueps) \D T_{k}(\ueps)
  dx \\
  + \int_{\Om} h_{n}'(\ueps) \aop (x,\ueps, \D \ueps) \D \ueps   T_{k}(\ueps)
  dx  
  +   \int_{\Om} h_{n}(\ueps)\Phieps (x, \ueps)  \D T_{k}(\ueps) dx \\
  + \int_{\Om} h_{n}'(\ueps)\Phieps (x, \ueps)  \D \ueps T_{k}(\ueps) dx =
  \int_\Om \feps T_{k}(\ueps) h_{n}(\ueps) dx.
\end{multline}
We now pass to the limit as $n$ goes to infinity. In view of the
definition of $h_{n}$ for any 
$n>k$ we have
\begin{equation*}
  \int_{\Om} h_{n}(\ueps) \aop (x,\ueps, \D \ueps) \D T_{k}(\ueps)
  dx =   \int_{\Om} \aop (x,\ueps, \D \ueps) \D T_{k}(\ueps)
  dx 
\end{equation*}
and
\begin{equation*}
  \int_{\Om} h_{n}(\ueps)\Phieps (x, \ueps)  \D T_{k}(\ueps) dx= \int_{\Om} \Phi
  (x, \ueps)  \D T_{k}(\ueps) dx .
\end{equation*}
Since $\ueps$ is finite almost everywhere in $\Omega$, the function
$h_{n}(\ueps)$ converges to $1$ in $L^{\infty}(\Omega)$ weak*, so that
\begin{equation*}
  \lim_{n\rightarrow +\infty}  \int_\Om \feps T_{k}(\ueps) h_{n}(\ueps) dx =  \int_\Om \feps T_{k}(\ueps)  dx.
\end{equation*}
Due to \eqref{def3}, we get
\begin{equation*}
  \lim_{n\rightarrow +\infty}  \int_\Omega h_{n}'(\ueps) \aop (x,\ueps, \D \ueps) \D \ueps   T_{k}(\ueps)
  dx  = 0.
\end{equation*}
It remains to control the behavior of the forth term to the right
hand side of \eqref{5eq1}. Since we have
\begin{equation*}
  \bigg| \int_{\Om} h_{n}'(\ueps)\Phieps (x, \ueps)  \D \ueps T_{k}(\ueps) dx \bigg|
  \leq
  \frac{k}{n} \int_{\Omega} |\Phieps(x,\ueps)|\times |\D T_{2n}(\ueps)| dx
\end{equation*}
recalling \eqref{ermk1} we obtain that
\begin{equation*}
  \lim_{n\rightarrow+\infty}  \bigg| \int_{\Om} h_{n}'(\ueps)\Phieps (x, \ueps)
  \D \ueps T_{k}(\ueps) dx \bigg| =0 .
\end{equation*}
It follows that passing to the limit as $n$ goes to infinity in
\eqref{5eq1} leads to
\begin{multline}
  \label{stab6}
  \int_{\Om}  \aop (x,\ueps, \D \ueps) \D T_{k}(\ueps)
  dx 
  +   \int_{\Om}\Phi (x, \ueps)  \D T_{k}(\ueps) dx\\
   =
  \int_\Om \feps T_{k}(\ueps) dx
\end{multline}
and then assumptions on $\aop$, $\Phieps$ and $\feps$ give \eqref{stab6a} and \eqref{stab6b}.

For the same reasons following Step 2 in the proof of
Theorem \ref{exist_renorm}, there exists a function $u$  such that, up to a subsequence still indexed by
$\varepsilon$, 
\begin{gather*}
  \label{stab_conv1}  u_\vp \rightarrow u \text{ a.e. in } \Om,
  \\
  \label{stab_conv2} T_k(u_\vp) \rightharpoonup T_k(u) \text{ weakly   in }
  W^{1,p}(\Om),
  \\
  \label{stab_conv3}   \aop (x,T_{k}(u_\vp),\D
  T_k(u_\vp))\rightharpoonup \sigma_k \text{ weakly  in } (L^{p'}(\Om))^N
  \quad\forall k>0,
\end{gather*}
where $\sigma_{k}$ belongs to $L^{p'}(\Omega)$ for any $k>0$.
\par

Using a similar process to one used to obtain \eqref{stab6} we get
\begin{multline}\label{stab7}
  \int_{\Om} \aop(x,u_\vp, \D u_\vp) \frac {\D u_\vp}{(1+|u_\vp|)^p}
  dx 
  +\int_{\Om} \Phi_\vp (x, u_\vp) \frac {\D u_\vp}{(1+|u_\vp|)^p}
  dx\\
   = \int_\Om f_\vp\Psi_p(u_\vp) dx,
\end{multline}   
where $\Psi_{p}(r)=\int_{0}^{r}\frac{1}{(1+|s|)^{p}} ds$. 
Therefore the arguments developed in Steps 3 and 4 imply that $u$ is
finite almost everywhere in $\Omega$ and lead to \eqref{stab6d}.
Because the sequel of the proof uses mainly admissible test function
in the renormalized formulation and the monotone character of the
operator we can repeat the same arguments to show that  $u$ is a
renormalized solution to \eqref{pb} with null median. In particular following
Steps 5 and 6 (see \eqref{4eq8} in the proof of Theorem \ref{exist_renorm})
allow to obtain  that \eqref{stab5} hold.
\end{proof}

Now we  prove that if $\aop(x,r,\xi)$ is a
classical Leray-Lions operator verifying \eqref{3eq3} and if $f\in
L^{q}$ with $q\leq (p^{*})'$ then any renormalized solution to
\eqref{pb} is also a weak solution to \eqref{pb} belonging to $W^{1,p}(\Omega)$.

\begin{proposition}\label{prop5}
    Assume that  \eqref{ell}, \eqref{mon}, \eqref{growthphi}, \eqref{comp} and
  \eqref{3eq3} hold. Let $u$ be a renormalized solution to
  \eqref{pb} with $\med(u)=0$. If $f\in L^{q}(\Omega)$ with $q\leq
  (p^{*})'$ if $N>p$ 
  and $q<+\infty$ if $N=p$ then $u$ belongs to $W^{1,p}(\Omega)$ and 
   \begin{gather*}
    \int_{\Omega} \aop(x,u,\nabla u)\nabla v dx +\int_{\Omega}
    \Phi(x,u) \nabla v dx =\int_{\Omega} fv  dx
  \end{gather*}
  for any $v\in W^{1,p}(\Omega)$.
\end{proposition}


\begin{proof}

Let $u$ be a renormalized solution to \eqref{pb}. We can proceed as in the proof of Theorem \ref{thstab} and we obtain 
\eqref{stab6}. Then we have
\begin{equation}
  \label{5eq2}
  \int_{\Om}  \aop (x,u, \D u) \D T_{k}(u)
  dx 
  +   \int_{\Om}\Phi (x, u)  \D T_{k}(u) dx =
  \int_\Om f T_{k}(u) dx.
\end{equation}
Using \eqref{ell}, \eqref{growthphi} and the regularity of $f$ we
obtain
\begin{equation*}
	\alpha  \int_{\Omega} |\nabla T_{k}(u)|^{p} dx \leq \int_{\Omega} c(x)
  (1+|u|^{p-1}) |\nabla T_{k}(u)| dx + 
  \|f\|_{L^{q}(\Omega)} \|T_{k}(u)\|_{L^{q'}(\Omega)}.
\end{equation*}
Let $R>0$ be a real number which will be chosen later and denote
\begin{equation*}
  E_{R}=\{ x\in\Omega\,;\, |u(x)|>R\}.
\end{equation*}
Using again $\med(T_{k}(u))=0$, Poincar\'e-Wirtinger inequality and
Sobolev embedding Theorem we have
\begin{equation*}
  \begin{split}
	  \alpha  \int_{\Omega} |\nabla T_{k}(u)|^{p} dx \leq  &  \int_{\Omega} c(x)
   |\nabla T_{k}(u)| dx +  \int_{E_{R}} c(x) |u|^{p-1}
   |\nabla T_{k}(u)| dx
   \\
   & {} +  \int_{\Omega\setminus E_{R}} c(x) |u|^{p-1}
   |\nabla T_{k}(u)| dx +
  \|f\|_{L^{q}(\Omega)} \|T_{k}(u)\|_{L^{q'}(\Omega)} .
  \end{split}
  \end{equation*}
  If follows that
  \begin{equation*}
	  \begin{split}
    \int_{\Omega} |\nabla T_{k}(u)|^{p} dx \leq {} & C \Big( \| \nabla T_{k}(u)\|_{(L^{p}(\Omega))^{N}} +  
  \|c\|_{L^{q}(E_{R})} \| \nabla T_{k}(u)\|_{(L^{p}(\Omega))^{N}}^{p}
  \\
  & {} + R^{p-1}  \| \nabla T_{R}(u)\|_{(L^{p}(\Omega))^{N}}\Big)
  \end{split}
\end{equation*}
where $C>0$ depends on $\alpha$, $f$, $N$, $p$, $\meas(\Omega)$, $c$ but is independent
of $k$.
Since $u$ is finite a.e. in $\Omega$, $\lim_{R\rightarrow
  +\infty}\meas(E_{R})=0$. By the equi-integrability of $c$ in
$L^{q}(\Omega)$ we can choose $R>0$ such that $C
\|c\|_{L^{q}(E_{R})}$ is sufficiently small enough so that
\begin{equation*}
  \int_{\Omega} |\nabla T_{k}(u)|^{p} dx \leq  
   C \Big( \| \nabla T_{k}(u)\|_{(L^{p}(\Omega))^{N}} 
  + R^{p-1}  \| \nabla T_{R}(u)\|_{(L^{p}(\Omega))^{N}}\Big)
\end{equation*}
where $C>0$ does not depend on $k$. It follows that
\begin{equation*}
   \int_{\Omega} |\nabla T_{k}(u)|^{p} dx \leq  C
\end{equation*}
where $C>0$ depends on $\alpha$, $f$, $N$, $p$, $\Omega$, $c$, $R$ but is independent
of $k$. Since $\med(T_k(u))=0$ Poincar\'e-Wirtinger inequality implies that
$T_k(u)$ is bounded in $W^{1,p}(\Omega)$ uniformly with respect to $k$.
Therefore we conclude that $u$ belongs to $W^{1,p}(\Omega)$.
\par

Using the renormalized formulation \eqref{def4} with $h=h_{n}$ and
passing to the limit as $n$ goes to infinity leads to 
\begin{gather}\label{5eq3}
  \int_{\Omega} \aop(x,u,\nabla u)\nabla v dx +\int_{\Omega}
  \Phi(x,u) \nabla v dx =\int_{\Omega} fv  dx
\end{gather}
for any $v\in L^{\infty}(\Omega)\cap  W^{1,p}(\Omega)$. Due to growth
assumptions \eqref{3eq3} on $\aop$ and \eqref{growthphi} on $\Phi$ we
deduce that $\aop(x,u,\nabla u)$ and $\Phi(x,u)$ belong to
$(L^{p'}(\Omega))^{N}$. It follows that \eqref{5eq3} holds for any
$v\in W^{1,p}(\Omega)$.
\end{proof}

\section{Operator with a zero order term}

In this section we consider Neumann problems which are similar to \eqref{pb0}
with a zero order term. Precisely 
   let us consider the following Neumann problem
  \begin{equation} \label{newpb}
      \begin{cases}
        \lambda(x,u) -\diw\left( \aop\left( x, u, \nabla u\right)+ \Phi (x,u)
        \right) =f   & \text{in}\ \Omega, \\
        \left( \aop\left( x, u, \nabla u\right)+ \Phi (x,u) \right)\cdot\underline n=0 & \text{on}\ \partial \Omega%
      \end{cases}%
  \end{equation}%
  where $\lambda : \Omega\times \R$ is a Carath\'eodory function
  verifying 
  \begin{gather}
    \label{zeroorder0}
    \lambda(x,s)s \geq 0, \\
    \label{zeororder1}
    \forall k>0,\ \exists c_{k}>0\text{ such that } |\lambda(x,s)|\leq
    c_{k} \quad\forall |s|\leq k,\text{ a.e. in }\Omega,
    \\
    \label{zeroorder2}
    \forall s\in\R\quad  |\lambda(x,s)|\geq g(s) \text{ a.e. in $\Omega$}
  \end{gather}
  where $g$ is function such that $\lim_{s\rightarrow \pm \infty}g(s)=
  +\infty$.
  \par
  
  If $f$ belongs to $L^{1}(\Omega)$ and without additional growth assumptions on
  $g$ we cannot expect to have in general a  solution (in whatever
  sense) lying in $L^{1}(\Omega)$ and then we have similar difficulties
  to deal with \eqref{newpb}. In particular the presence of $\lambda(x,u)$ does
  not help to deal with the term $-\diw(\Phi(x,u))$ and we cannot follow the
  approach of \cite{AMST97,Droniou00,DV,Prignet97} which use the mean 
  value. However the ``median'' tool and some 
  modifications of the proof of Theorem \ref{exist_renorm} allow to show
  that there exists at least a renormalized solution to
  \eqref{newpb}:
  \begin{theorem} \label{exist_renorm_zero_order_term}
 Assume \eqref{ell}--\eqref{comp} and \eqref{zeroorder0}--\eqref{zeroorder2}. If the datum $f$ belongs to $ L^{1}(\Om) $ then there exists at least one renormalized solution $u$ to problem \eqref{newpb}. 
\end{theorem}

\begin{proof}[Sketch of proof]

 As in Theorem \ref{exist_weak}, a fixed point theorem
  and classical results of Leray-Lions give the existence of 
  $\ueps$ belonging to $W^{1,p}(\Omega)$ verifying
  \begin{multline}
    \varepsilon   \int_{\Omega}|\ueps|^{p-2} \ueps v dx +
    \int_{\Omega}\lambda(x,T_{1/\varepsilon}(\ueps)) v dx \\
    + 
    \int_{\Om}  \aop(x,T_{1/\varepsilon}(u_\vp),\D u_\vp)\D v dx
    \\
    +\int_{\Om} \Phi(x,T_{1/\varepsilon}(u_\vp))\D v  dx
    =\int_{\Om}T_{1/\varepsilon}(f) v dx 
    \label{appr_newpb_eps}
  \end{multline}
  for any $v$ lying in $W^{1,p}(\Omega)$. Due to the zero order term
  $\varepsilon |\ueps|^{p-2} \ueps +\lambda(x,T_{1/\varepsilon}(\ueps))$
  in the equation, we do not need any compatibility condition on
  $f$. The counter part is that we cannot expect to have (or to fix)
  $\med(\ueps)=0$ and then it yields another difficulties. In particular
  Steps 3 and 4 (see the proof of Theorem \ref{exist_renorm}) which use strongly the fact that the solution has a null
  median should be adapted in the case of the approximated problem
  \eqref{appr_newpb_eps}. Step 2 is unchanged and we have the following and
  additional estimate
  \begin{equation}
    \label{neu00}
    T_{1/\varepsilon}(g(\ueps)) \text{ bounded in } L^{1}(\Omega).
  \end{equation}
  Due to the behavior at infinity of the function $g$ we deduce that
  \begin{gather}
    \label{neu01}
    \lim_{A\rightarrow +\infty}\sup_{\varepsilon>0
    } \meas\{ x\in \Omega;\,
    |\ueps(x)|>A\} =0,
    \\
    \label{neu02}
    \forall \varepsilon>0 \quad |\med(\ueps)|\leq M
  \end{gather}
  where $M$ is a positive real number independent of $\varepsilon$. It
  follows (after extracting appropriate subsequence, see Step 2) that
  there exists a measurable function $u$ which  is finite almost
  everywhere in $\Omega$ such that
  \begin{gather*}
    u_\vp \rightarrow u \text{ a.e. in } \Om,
    \\
    T_k(u_\vp) \rightharpoonup T_k(u) \text{ weakly   in }
    W^{1,p}(\Om), \quad\forall k>0.
  \end{gather*}
  Step 4 which is crucial
  in dealing with renormalized solutions consists here in
  proving that
  \begin{equation} \label{neu03}
    \lim_{n\rightarrow +\infty}\limsup_{\varepsilon\rightarrow 0}
    \frac{1}{n} \int_{\Omega}
    \aop(x,u_{\varepsilon},\nabla u_{\varepsilon}) \nabla
    T_{n}(u_{\varepsilon}) dx = 0
  \end{equation}
  using the test function $T_{n}(\ueps)$ in \eqref{appr_newpb_eps}. Due to the sign
  condition \eqref{zeroorder0} the  contribution of the zero order terms
  \begin{equation*}
    \varepsilon   \int_{\Omega}|\ueps|^{p-2} \ueps T_{n} (\ueps)  dx +
    \int_{\Omega}\lambda(x,T_{1/\varepsilon}(\ueps)) T_{n}(\ueps) dx 
  \end{equation*}
  is positive.  It follows that the inequality \eqref{eqog1} holds: 
  \begin{equation*}
    \begin{split}
      \frac{1}{n}\int_{\Om} \aop (x,u_\vp,\D u_\vp) & \D T_n(u_\vp) dx
      \leq \frac{1}{n} \int_{\Om}|T_{1/\varepsilon}(f)|\times  |T_n(u_\vp)| dx
      \\
      & {} + \frac{1}{n}\int_{\Om} c(x) (1+|T_n (u_\vp)|^{p-1}) |\D T_n(u_\vp)| dx.
    \end{split}
  \end{equation*}

  Because we do not have in the present case 
  the property $\med(\ueps)=\med(T_{n}(\ueps))=0$
  we have to modify the estimate of the term
  \begin{equation}\label{neu04}
    \frac{1}{n} \int_{\Omega} c(x) (1+|T_{n}(\ueps)|^{p-1}) |\nabla
    T_{n}(\ueps)| dx.
  \end{equation}

  In view of \eqref{neu02} we have for any $n>0$ and for any
  $\varepsilon>0$ $|\med(T_{n}(\ueps))| \leq M$.
  It follows that by writing $T_{n}(\ueps)=T_{n}(\ueps)-\med(T_{n}(\ueps)) +
  \med(T_{n}(\ueps))$ we obtain
  \begin{equation}\label{neu05}
    \begin{split} 
      \frac{1}{n} \int_{\Omega} c(x) &(1+|T_{n}(\ueps)|^{p-1})  |\nabla
      T_{n}(\ueps)| dx  \\
      & \leq
      \frac{C}{n} \int_{\Omega} c(x)
      (1+|T_{n}(\ueps)-\med(T_{n}(\ueps))|^{p-1}) 
      |\nabla   T_{n}(\ueps)| dx 
    \end{split}
  \end{equation}
  where $C>0$ is a constant independent of $\varepsilon$ and $n$.
  Poincar\'e-Wirtinger inequality  \eqref{poincare}, similar
  arguments to the ones developed in Step 4 and \eqref{neu02} then allow
  conclude that \eqref{neu03} holds.
  \par
  As far as Step 5 is concerned, it is sufficient to remark that the
  Lebesgue Theorem yields that
  \begin{gather*}
    \lim_{n\rightarrow +\infty} \lim_{\varepsilon\rightarrow0 }\varepsilon 
    \int_{\Omega} h_{n}(\ueps) |\ueps|^{p-2} \ueps (T_{k}
    (\ueps)-T_{k}(u))  dx =0 
    \\
    \lim_{n\rightarrow +\infty} \lim_{\varepsilon\rightarrow0 }
    \int_{\Omega} h_{n}(\ueps) \lambda(x,T_{1/\varepsilon}(\ueps)) (T_{k}
    (\ueps)-T_{k}(u)) dx =0.
  \end{gather*}
  Since Step 6 remains unchanged, in Step 7 we pass
  to the limit as $\varepsilon$ goes to zero in 
  \begin{multline*}
    \varepsilon 
    \int_{\Omega} h(\ueps) |\ueps|^{p-2} \ueps \varphi dx +
    \int_{\Omega} h(\ueps) \lambda(x,T_{1/\varepsilon}(\ueps)) \varphi dx 
    \\
    + \int_{\Om} h(u_\vp) \aop (x,u_\vp, \D u_\vp) \D\varphi
    dx + \int_{\Om} h'(u_\vp) \aop (x,u_\vp, \D u_\vp) \D u_\vp   \varphi
    dx  \\ 
    +   \int_{\Om} h(u_\vp)\Phi (x, u_\vp)  \D\varphi dx + \int_{\Om} h'(u_\vp)\Phi (x, u_\vp)  \D u_\vp \varphi dx \\
    = {} \int_\Om T_{1/\varepsilon}(f)\varphi h(u_\vp) dx
  \end{multline*}
  where $h$ is a Lispchitz continuous function with compact support and
  where $\varphi$ lies in $W^{1,p}(\Omega)\cap L^{\infty}(\Omega)$. Since the
  Lebesgue Theorem gives that
  \begin{gather*}
    \lim_{\varepsilon\rightarrow 0}  \varepsilon 
    \int_{\Omega} h(\ueps) |\ueps|^{p-2} \ueps \varphi dx =0
    \\
    \lim_{\varepsilon\rightarrow 0}  \int_{\Omega} h(\ueps)
    \lambda(x,T_{1/\varepsilon}(\ueps)) \varphi dx = \int_{\Omega}
    h(u) \lambda(x,u) \varphi dx
  \end{gather*}
  the attentive reader may convince by himself that we obtain the
  existence of a renormalized solution to equation \eqref{newpb}.
\end{proof}

\section*{Acknowledgement}
This work was done during the visits made by the first and the third authors to
Laboratoire de Math\'ematiques ``Rapha\"el Salem'' de l'Universit\'e
de Rouen and by the second author to 
 Dipartimento di Matematica e
Applicazioni ``R. Caccioppoli'' dell' Universit\`a degli Studi di
Napoli ``Federico II''. 
Hospitality and support of all these
institutions are gratefully acknowledged.

\bibliography{bgm1}

\def\cprime{$'$}
\begin{thebibliography}{10}

\bibitem{ACMM}
A.~Alvino, A.~Cianchi, V.~G. Maz'ya, and A.~Mercaldo.
\newblock Well-posed elliptic {N}eumann problems involving irregular data and
  domains.
\newblock {\em Ann. Inst. H. Poincar\'e Anal. Non Lin\'eaire},
  27(4):1017--1054, 2010.

\bibitem{AM1}
A.~Alvino and A.~Mercaldo.
\newblock Nonlinear elliptic problems with {$L^1$} data: an approach via
  symmetrization methods.
\newblock {\em Mediterr. J. Math.}, 5(2):173--185, 2008.

\bibitem{AMST97}
F.~Andreu, J.~M. Maz{\'o}n, S.~Segura~de Le{\'o}n, and J.~Toledo.
\newblock Quasi-linear elliptic and parabolic equations in {$L^1$} with
  nonlinear boundary conditions.
\newblock {\em Adv. Math. Sci. Appl.}, 7(1):183--213, 1997.

\bibitem{BeGu}
M.~Ben Cheikh~Ali and O.~Guib{\'e}.
\newblock Nonlinear and non-coercive elliptic problems with integrable data.
\newblock {\em Adv. Math. Sci. Appl.}, 16(1):275--297, 2006.

\bibitem{BBGGPV}
P.~B{\'e}nilan, L.~Boccardo, T.~Gallou{\"e}t, R.~Gariepy, M.~Pierre, and J.~L.
  V{\'a}zquez.
\newblock An {$L^1$}-theory of existence and uniqueness of solutions of
  nonlinear elliptic equations.
\newblock {\em Ann. Scuola Norm. Sup. Pisa Cl. Sci. (4)}, 22(2):241--273, 1995.

\bibitem{BMMP1}
M.~F. Betta, A.~Mercaldo, F.~Murat, and M.~M. Porzio.
\newblock Existence and uniqueness results for nonlinear elliptic problems with
  a lower order term and measure datum.
\newblock {\em C. R. Math. Acad. Sci. Paris}, 334(9):757--762, 2002.

\bibitem{BMMP}
M.~F. Betta, A.~Mercaldo, F.~Murat, and M.~M. Porzio.
\newblock Existence of renormalized solutions to nonlinear elliptic equations
  with a lower-order term and right-hand side a measure.
\newblock {\em J. Math. Pures Appl. (9)}, 82(1):90--124, 2003.
\newblock Corrected reprint of J. Math. Pures Appl. (9) {{\bf{8}}1} (2002), no.
  6, 533--566 [ MR1912411 (2003e:35075)].

\bibitem{BG1}
L.~Boccardo and T.~Gallou{\"e}t.
\newblock Nonlinear elliptic and parabolic equations involving measure data.
\newblock {\em J. Funct. Anal.}, 87(1):149--169, 1989.

\bibitem{BG2}
L.~Boccardo and T.~Gallou{\"e}t.
\newblock Nonlinear elliptic equations with right-hand side measures.
\newblock {\em Comm. Partial Differential Equations}, 17(3-4):641--655, 1992.

\bibitem{BOP}
L.~Boccardo, L.~Orsina, and A.~Porretta.
\newblock Some noncoercive parabolic equations with lower order terms in
  divergence form.
\newblock {\em J. Evol. Equ.}, 3(3):407--418, 2003.
\newblock Dedicated to Philippe B{\'e}nilan.

\bibitem{Ciabr}
J.~Chabrowski.
\newblock On the {N}eumann problem with {$L^1$} data.
\newblock {\em Colloq. Math.}, 107(2):301--316, 2007.

\bibitem{DMOP}
G.~Dal~Maso, F.~Murat, L.~Orsina, and A.~Prignet.
\newblock Renormalized solutions of elliptic equations with general measure
  data.
\newblock {\em Ann. Scuola Norm. Sup. Pisa Cl. Sci. (4)}, 28(4):741--808, 1999.

\bibitem{Aglio}
A.~Dall'Aglio.
\newblock Approximated solutions of equations with {$L^1$} data. {A}pplication
  to the {$H$}-convergence of quasi-linear parabolic equations.
\newblock {\em Ann. Mat. Pura Appl. (4)}, 170:207--240, 1996.

\bibitem{Rako}
A.~Decarreau, J.~Liang, and J.-M. Rakotoson.
\newblock Trace imbeddings for {$T$}-sets and application to
  {N}eumann-{D}irichlet problems with measures included in the boundary data.
\newblock {\em Ann. Fac. Sci. Toulouse Math. (6)}, 5(3):443--470, 1996.

\bibitem{Droniou00}
J.~Droniou.
\newblock Solving convection-diffusion equations with mixed, {N}eumann and
  {F}ourier boundary conditions and measures as data, by a duality method.
\newblock {\em Adv. Differential Equations}, 5(10-12):1341--1396, 2000.

\bibitem{DV}
J.~Droniou and J.-L. V{\'a}zquez.
\newblock Noncoercive convection-diffusion elliptic problems with {N}eumann
  boundary conditions.
\newblock {\em Calc. Var. Partial Differential Equations}, 34(4):413--434,
  2009.

\bibitem{FM2}
V.~Ferone and A.~Mercaldo.
\newblock A second order derivation formula for functions defined by integrals.
\newblock {\em C. R. Acad. Sci. Paris S\'er. I Math.}, 326(5):549--554, 1998.

\bibitem{FM}
V.~Ferone and A.~Mercaldo.
\newblock Neumann problems and {S}teiner symmetrization.
\newblock {\em Comm. Partial Differential Equations}, 30(10-12):1537--1553,
  2005.

\bibitem{GM2}
O.~Guib{\'e} and A.~Mercaldo.
\newblock Existence and stability results for renormalized solutions to
  noncoercive nonlinear elliptic equations with measure data.
\newblock {\em Potential Anal.}, 25(3):223--258, 2006.

\bibitem{GM1}
O.~Guib{\'e} and A.~Mercaldo.
\newblock Existence of renormalized solutions to nonlinear elliptic equations
  with two lower order terms and measure data.
\newblock {\em Trans. Amer. Math. Soc.}, 360(2):643--669 (electronic), 2008.

\bibitem{LL}
J.~Leray and J.-L. Lions.
\newblock Quelques r\'esulatats de {V}i\v sik sur les probl\`emes elliptiques
  non lin\'eaires par les m\'ethodes de {M}inty-{B}rowder.
\newblock {\em Bull. Soc. Math. France}, 93:97--107, 1965.

\bibitem{Lions}
J.-L. Lions.
\newblock {\em Quelques m\'ethodes de r\'esolution des probl\`emes aux limites
  non lin\'eaires}.
\newblock Dunod, 1969.

\bibitem{LM}
P.L. Lions and F.~Murat.
\newblock Sur les solutions renormalis\'ees d'{\'e}quations elliptiques non
  lin\'eaires.
\newblock In {\em manuscript}.

\bibitem{murat94}
F.~Murat.
\newblock Equations elliptiques non lin\'eaires avec second membre ${L}^1$ ou
  mesure.
\newblock In {\em Compte Rendus du 26\`eme Congr\`es d'Analyse Num\'erique},
  les Karellis, 1994.

\bibitem{Prignet97}
A.~Prignet.
\newblock Conditions aux limites non homog\`enes pour des probl\`emes
  elliptiques avec second membre mesure.
\newblock {\em Ann. Fac. Sci. Toulouse Math. (6)}, 6(2):297--318, 1997.

\bibitem{St}
G.~Stampacchia.
\newblock Le probl\`eme de {D}irichlet pour les \'equations elliptiques du
  second ordre \`a coefficients discontinus.
\newblock {\em Ann. Inst. Fourier (Grenoble)}, 15(fasc. 1):189--258, 1965.

\bibitem{Z}
W.~P. Ziemer.
\newblock {\em Weakly differentiable functions}, volume 120 of {\em Graduate
  Texts in Mathematics}.
\newblock Springer-Verlag, New York, 1989.
\newblock Sobolev spaces and functions of bounded variation.

\end{thebibliography}
\bibliographystyle{plain}

\end{document}